\documentclass[final]{elsarticle}

\usepackage{hyperref}
\hypersetup{
	colorlinks=true,
	linkcolor=blue,
	filecolor=blue,      
	urlcolor=red,
	citecolor=green,
}
\usepackage{graphicx}
\usepackage{subfigure}
\usepackage{amsmath}
\usepackage{amssymb}
\usepackage{amsthm}
\usepackage{geometry}
\usepackage[ruled,linesnumbered]{algorithm2e}

\DeclareMathOperator{\N}{\mathbb{N}}

\DeclareMathOperator{\R}{\mathbb{R}}
\DeclareMathOperator{\Span}{\text{span}}

\newcommand{\wtV}{\widetilde{V}}
\newcommand{\whV}{\widehat{V}}
\newcommand{\Pd}{\mathcal{P}(d)}
\newcommand{\calC}{\mathcal{C}}
\newcommand{\calA}{\mathcal{A}}
\newcommand{\calS}{\mathcal{S}}
\newcommand{\calU}{\mathcal{U}}
\newcommand{\calV}{\mathcal{V}}
\newcommand{\calW}{\mathcal{W}}

\newcommand{\bbH}{\mathbb{H}}
\DeclareMathOperator{\lrlex}{<^{\text{lrlex}}}
\DeclareMathOperator{\Log}{logical}
\DeclareMathOperator{\powerproduct}{\widehat{\circledast}}
\newcommand{\defas}{\overset{\text{def}}{=}}
\DeclareMathOperator{\spa}{\text{spa}}
\DeclareMathOperator{\nnz}{\text{nnz}}

\theoremstyle{plain}
\newtheorem{thm}{Theorem}[section]
\newtheorem{lem}[thm]{Lemma}
\newtheorem{prop}[thm]{Proposition}
\newtheorem{cor}[thm]{Corollary}

\theoremstyle{definition}
\newtheorem{defn}{Definition}[section]
\newtheorem{conj}{Conjecture}[section]
\newtheorem{exmp}{Example}[section]

\theoremstyle{remark}
\newtheorem*{rem}{Remark}

\usepackage{CJKutf8}
\usepackage[breakable]{tcolorbox}
\newtcolorbox{mybox}{breakable,coltitle=black,colbacktitle=gray!20,colframe=black,colback=white,coltext=black}

\graphicspath{{figs/}}

\begin{document}

\begin{frontmatter}

\title{Power-product matrix: nonsingularity, sparsity and determinant \tnoteref{mytitlenote}}
\tnotetext[mytitlenote]{The authors are supported by the National Natural Science Foundation of China (Grant 11601327).}


\author[mymainaddress,mysecondaryaddress]{Yi-Shuai Niu\corref{mycorrespondingauthor}}
\cortext[mycorrespondingauthor]{Corresponding author}
\ead{niuyishuai@sjtu.edu.cn}
\author[mymainaddress]{Hu Zhang}
\ead{jd\_luckymath@sjtu.edu.cn}

\address[mymainaddress]{School of Mathematical Sciences, Shanghai Jiao Tong University, Shanghai 200240, China}
\address[mysecondaryaddress]{SJTU-Paristech Elite Institute of Technology, Shanghai Jiao Tong University, Shanghai 200240, China}

\begin{abstract}
We prove the nonsingularity of a class of integer matrices $V(n,d)$, namely power-product matrix, for positive integers $n$ and $d$. Some technical proofs are mainly based on linear algebra and enumerative combinatorics, particularly the generating function method and involution principle. We will show that the matrix $V(n,d)$ is nonsingular for all positive integers $n$ and $d$, and often with sparse structure. Special attention is given to the computation of the determinant $V(2,d)$ with $d\in \N^*$. 
\end{abstract}


\end{frontmatter}


\section{Introduction}
In this paper, we are interested in the properties (nonsingularity, sparsity and determinant) of a special structured matrix, namely power-product matrix, which is defined as follows: Let $n\in \mathbb{N}^*$ and $d\in \mathbb{N}^*$, consider the problem of putting $d$ balls into $n$ bins (empty bin is allowed), then the totally number of possible choices is denoted by $s_{n,d}$ which equals to ${d+n-1 \choose d}$. Particularly, if $n=d$, we write $s_n=s_{n,n}$. The \emph{power-product matrix} is defined by 
\begin{equation}\label{eq:V(n,d)}
  V(n,d) \defas \begin{bmatrix}
(\alpha^1)^{\alpha^1} & \cdots & (\alpha^1)^{\alpha^{s_{n,d}}}\\
(\alpha^2)^{\alpha^1} & \cdots & (\alpha^2)^{\alpha^{s_{n,d}}}\\
\vdots & \vdots & \vdots\\
(\alpha^{s_{n,d}})^{\alpha^1} & \cdots & (\alpha^{s_{n,d}})^{\alpha^{s_{n,d}}}
\end{bmatrix},  
\end{equation}
where $\alpha^i=(\alpha^i_1,\ldots, \alpha^i_n)\in \N^n$ is one possible choice of putting $d$ balls into $n$ bins; $(\alpha^i)^{\alpha^j}$ is the power-product of $\alpha^i$ and $\alpha^j$ defined by $$(\alpha^i)^{\alpha^j} \defas \displaystyle\prod_{k=1}^{n}(\alpha^i_k)^{\alpha^j_k},$$ under the assumption that $0^0=1$. Let us denote $$B(n,d) = \{ \alpha^1,\ldots,\alpha^{s_{n,d}}\}$$
or in matrix form $\left[
 \alpha^1 | \cdots | \alpha^{s_{n,d}} \right]$ with $n$ rows and $s_{n,d}$ columns where $\alpha^i$ is a column vector, we can define in general the \emph{power-product} of two matrices as in the next definition:
\begin{defn}[Power-product]
Let $A = [a_{i,j}]_{n\times m}$ be an $n\times m$ matrix and $B=[b_{i,j}]_{m\times q}$ be an $m\times q$ matrix, then \emph{$A$ power-product $B$}, denoted by $A\powerproduct B$, is an $n\times q$ matrix defined by 
\begin{equation}
    \label{eq:power-product}
A\powerproduct B = \left[\prod_{k=1}^{m} a_{i,k}^{b_{k,j}}\right]_{n\times q}.
\end{equation}
\end{defn}
Obviously, the definition of the power-product is quite similar to the matrix multiplication, i.e., $$A\cdot B = \left[\sum_{k=1}^{m} a_{i,k} \times b_{k,j}\right]_{n\times q},$$ 
where  $a_{i,k}\times {b_{k,j}}$ and $\sum$ in matrix multiplication is replaced by  $a_{i,k}^{b_{k,j}}$ and $\prod$ in power-product. As in matrix multiplication, the power-product is also non-commutative. Suppose that $\alpha$ is a column vector and denote $\alpha^{\top}$ as the transpose of $\alpha$, then we can use the power-product notation to rewrite $(\alpha^i)^{\alpha^j}$ as $(\alpha^i)^{\top} \powerproduct \alpha^j$, and the power-product matrix $V(n,d)$ as  
$$V(n,d) = B(n,d)^{\top} \circledast B(n,d).$$

\begin{exmp}
Let $n=3$ and $d=2$, then the matrix $B(3,2)$ is
$$B(3,2) = \begin{bmatrix}
 2 & 0 & 0 & 1 & 1 & 0\\
 0 & 2 & 0 & 1 & 0 & 1\\
 0 & 0 & 2 & 0 & 1 & 1
\end{bmatrix}.$$
The elements in the first row of the matrix $V(3,2)$ are computed as:
$$(\alpha^1)^{\alpha^1} = (2,0,0)^{(2,0,0)} = 2^2\times 0^0\times 0^0 = 4,$$
$$(\alpha^1)^{\alpha^2} = (2,0,0)^{(0,2,0)} = 2^0\times 0^2\times 0^0 = 0,$$
and so on. Thus, we have the power-product matrix $V(3,2)$ as:
$$V(3,2) = \begin{bmatrix}
4&0&0&0&0&0\\
    0&4&0&0&0&0\\
    0&0&4&0&0&0\\
    1&1&0&1&0&0\\
    1&0&1&0&1&0\\
    0&1&1&0&0&1\\
\end{bmatrix}.
$$
It is easy to verify that $V(3,2)$ is nonsingular since it is a triangular matrix with non-zero diagonal elements, so that its eigenvalues are $1$ of multiplicity $3$ and $4$ with multiplicity $3$, and the determinant is $64$. The number of non-zero elements in $V(3,2)$ is $12$ among totally $36$ elements, so that the sparsity is about $66.67\%$.
\end{exmp}
Note that $V(n,d)$ is not always triangular and its nonsingularity is non-trivial at all. Moreover, the order of elements in the set $B(n,d)$ will not change the nonsingularity and the determinant of the matrix $V(n,d)$, because any permutation of elements in $B(n,d)$ leads to a permuted matrix of $V(n,d)$ (both in rows and columns) which will not change the nonsingularity and determinant.

In this paper, we will focus on the proof of the nonsingularity of $V(n,d)$ for all positive integers $n$ and $d$, and investigate some properties of this matrix, including the sparsity and the determinant computation. We will show that: (1) $V(n,d)$ is nonsingular for all positive integers $n$ and $d$; (2) $V(n,d)$ is often with sparse structure; (3) the formulation for computing the determinant of $V(n,d)$ with special $n$ and $d$ is established. A conjecture for general formulation of $\det V(n,d)$ with all positive integers $n$ and $d$ is proposed which deserves more attention in the future.

\section{Application of the power-product matrix}
The power-product matrix $V(n,d)$ arise in many applications. For example, in polynomial representation, it can be used to generate a convex multi-variate polynomial basis for the polynomial vector space $\R_d[x]$ (all real polynomials of variable $x\in \R^n$ and of degree $d$). Particularly, a so called Difference-of-Convex-Sums-of-Squares (DC-SOS) decomposition of polynomials defined in \cite{niu2018dcsos} can be established in this basis, which helps to reformulate any polynomial optimization problem as DC (difference-of-convex) programming problem, and the later one can be investigated using powerful theories and algorithms in the field of DC programming. In another view, This matrix is also related to construct a class of power sum representation for polynomials (see, e.g., \cite{lee2016power,fischer1994sums}).

To see this, let $x\in \R^n$, the multinomial equation reads 
\begin{equation}
    \label{eq:multinomial}
    \left(\sum_{i=1}^{n} x_i\right)^d = \sum_{\alpha\in B(n,d)} {d \choose \alpha} x^{\alpha},
\end{equation}
where ${d\choose \alpha} \defas \frac{d!}{\alpha_1!\alpha_2!\cdots\alpha_n!}$. Then, we get for all $\alpha\in B(n,d)$ that

$$\langle\alpha, x\rangle^d  =\left( \sum_{i=1}^{n} \alpha_i x_i \right)^d \overset{y_i=\alpha_ix_i}{=} \left( \sum_{i=1}^{n} y_i \right)^d \overset{\eqref{eq:multinomial}}{=} \sum_{\beta\in B(n,d)} {d \choose \beta} y^{\beta} \overset{y_i=\alpha_ix_i}{=} \sum_{\beta\in B(n,d)} {d \choose \beta} \alpha^{\beta} x^{\beta}.$$
It follows that
\begin{equation}\label{eq:iden}
\begin{bmatrix}
\langle\alpha^1, x\rangle^d \\
\langle\alpha^2, x\rangle^d \\
\vdots \\
\langle\alpha^{s_{n,d}}, x\rangle^d
\end{bmatrix}
=
\whV(n,d) \cdot \begin{bmatrix}
 x^{\alpha^1}\\
 x^{\alpha^2}\\
 \vdots\\
 x^{\alpha^{s_{n,d}}}
\end{bmatrix}.
\end{equation}
where 
\begin{equation}\label{eq:whV(n,d)}
\whV(n,d) = \begin{bmatrix}
 \tbinom{d}{\alpha^1} (\alpha^1)^{\alpha^1} &  \cdots & \tbinom{d}{\alpha^{s_{n,d}}} (\alpha^1)^{\alpha^{s_{n,d}}}\\
 \tbinom{d}{\alpha^1} (\alpha^2)^{\alpha^1} &  \cdots & \tbinom{d}{\alpha^{s_{n,d}}} (\alpha^2)^{\alpha^{s_{n,d}}}\\
  \vdots& \vdots&\vdots\\
  \tbinom{d}{\alpha^1} (\alpha^{s_{n,d}})^{\alpha^1}&  \cdots & \tbinom{d}{\alpha^{s_{n,d}}} (\alpha^{s_{n,d}})^{\alpha^{s_{n,d}}}
\end{bmatrix}. 
\end{equation}
Interestingly, $\whV(n,d)$ enjoys the identify that \begin{equation}
    \label{eq:relationVvsVh}
\whV(n,d) = V(n,d) \cdot \begin{bmatrix}
 \tbinom{d}{\alpha^1}\\
 & \ddots\\
 && \tbinom{d}{\alpha^{s_{n,d}}}
\end{bmatrix},
\end{equation}
which is nonsingular if and only if $V(n,d)$ is nonsingular.  Moreover, 
$$\det \whV(n,d) = \left(\prod_{i=1}^{m} \tbinom{d}{\alpha^i}\right) \det V(n,d).$$
Therefore, if we can prove that $V(n,d)$ is nonsingular for all positive integer couples $(n,d)$, then the set of polynomials $\mathcal{B}=\{ \langle\alpha^1, x\rangle^d , \ldots, \langle\alpha^{s_{n,d}}, x\rangle^d  \}$ is a basis of $\bbH_d[x]$ (all real homogeneous polynomials of variables $x\in \R^n$ and of degree $d$), because the set $\{ x^{\alpha^1}, \ldots, x^{\alpha^{s_{n,d}}}\}$ is a canonical basis of $\bbH_d[x]$. Furthermore, if $d$ is even, then $\mathcal{B}$ consists of convex and square polynomials, namely CSOS (convex-sums-of-squares) polynomials, which results that any even degree polynomial in $\bbH_d[x]$ can be presented in the CSOS basis $\mathcal{B}$. Similarly, in case of odd degree homogeneous polynomial, it can be formulated as an even degree homogeneous polynomial by multiplying a new variable which will be fixed to $1$ after representation. 

The statement: ``The set $\mathcal{B}$ is a CSOS polynomial basis of $\bbH_d[x]$ for even degree $d$" is first proposed as a conjecture in \cite{niu2018dcsos}, which will be proved in this paper as well. 

\section{Nonsingularity of $V(n,d)$}
In this section, we will focus on the nonsingularity of $\whV(n,d)$ defined in \eqref{eq:whV(n,d)} and $V(n,d)$ defined in \eqref{eq:V(n,d)} using techniques in enumerative combinatorics (see e.g., \cite{stanley2011enumerative}). As a common notation, let us denote the coefficient of $x^n$ in $F(x)=\sum_{n\ge0}a_nx^n$ as $$a_n=[x^n]F(x).$$ 
\begin{lem}\label{lem:app1}
Let $m\ge k \ge 0$, then for $n>m$, we have
\begin{equation}\label{eq:lem:app1}
\sum_{k=0}^m\ \ \sum_{a_1+a_2+\cdots+a_k=m,\atop \forall a_i\in\N^*}\dfrac{(-n)^k}{k!a_1a_2\cdots a_k}=(-1)^m\dbinom{n}{m},
\end{equation}
\end{lem}
\begin{proof}
Let $m\ge k \ge 0$, we have 
\begin{equation*}
    \begin{split}
        \sum_{a_1+\cdots+a_k=m,\atop \forall a_i\in\N^*}\dfrac{(zy)^k}{k!a_1a_2\cdots a_k}=& [x^m]\sum_{m\ge k}\sum_{a_1+\cdots+a_k=m,\atop \forall a_i\in\N^*}\dfrac{(zy)^kx^m}{k!a_1a_2\cdots a_k} \\
        =& [x^m]\sum_{m\ge k}\sum_{a_1+\cdots+a_k=m,\atop \forall a_i\in\N^*}\dfrac{(zy)^k}{k!}\dfrac{x^{a_1}}{a_1}\cdots \dfrac{x^{a_k}}{a_k}\\
        =& [x^m]\sum_{m\ge k}\dfrac{(zy)^k}{k!}\sum_{a_1+\cdots+a_k=m,\atop \forall a_i\in\N^*}\dfrac{x^{a_1}}{a_1}\cdots \dfrac{x^{a_k}}{a_k}\\
        =&[x^m]\dfrac{(zy)^k}{k!}(-\ln{(1-x)})^k .
    \end{split}
\end{equation*}
Thus, 
\begin{equation*}
    \begin{split}
        \sum_{a_1+\cdots+a_k=m,\atop \forall a_i\in\N^*}\dfrac{z^k}{k!a_1a_2\cdots a_k}
        =&[x^my^k]\dfrac{(zy)^k}{k!}(-\ln{(1-x)})^k\\
        =&[x^my^k]\sum_{i\ge 0}\dfrac{(-zy\ln{(1-x)})^i}{i!}\\
        =&[x^my^k]\exp{(-zy\ln{(1-x)})}\\
        =&[x^my^k](1-x)^{-zy}\\
        =&[y^k](-1)^m\dbinom{-zy}{m}\\
        =&\dfrac{z^k}{m!}\sum_{S\subseteq[m-1],|S|=m-k}\prod_{i\in S}i\\
        =&\dfrac{z^k}{m!}c(m,k),
    \end{split}
\end{equation*}

where $c(m,k)$ is the well-known signless stirling number of the first kind (see, e.g., \cite{stanley2011enumerative}). Hence,
$$\sum_{k=0}^m\sum_{a_1+\cdots+a_k=m,\atop \forall a_i\in\N^*}\dfrac{z^k}{k!a_1a_2\cdots a_k}=\sum_{k=1}^m\dfrac{z^k}{m!}c(m,k)=\dfrac{1}{m!}z(z+1)\cdots(z+m-1).$$

Taking $z=-n$ with $n>m$, we have 
$$\sum_{k=0}^m\sum_{a_1+\cdots+a_k=m,\atop \forall a_i\in\N^*}\dfrac{(-n)^k}{k!a_1a_2\cdots a_k}=\dfrac{1}{m!}(-n)(-n+1)\cdots(-n+m-1)=(-1)^m\dbinom{n}{m}.$$
\end{proof}
Note that the method used in the proof of Lemma \ref{lem:app1} is called the generating function approach, which is widely used in enumerative conbinatorics, see e.g., \cite[Chapter 1]{stanley2011enumerative}.

\begin{lem}\label{lem:app2}
Let $1\leq r \leq n$ and $b=(b_1,\cdots,b_r)\in (\N^*)^r.$ Then 
\begin{equation}\label{eq:lem:app2}
    \begin{split}
        &\sum_{k=r}^n\displaystyle\sum_{a_1+\cdots +a_k=n,\atop \forall a_i\in \N^*}(-1)^{n-k}\dfrac{\ \ (n-k)!}{\ \ n!n^{n-k}a_1\cdots a_k}\dbinom{n-r}{k-r}a_1^{b_1}\cdots a_r^{b_r}\\
        =& \dfrac{(-1)^r(n-r)!}{n!n^{n-r}}\sum_{s=r}^n(-1)^s\dbinom{n}{s}\left(\displaystyle\sum_{a_1+\cdots a_r=s,\atop \forall a_i\in \N^*}a_1^{b_1-1}\cdots a_r^{b_r-1}\right).
    \end{split}
\end{equation}
\end{lem}
\begin{proof}
Let $1\leq r\leq k\leq n$ and $(b_1,\cdots,b_r)\in (\N^*)^r$, then
\begin{equation}\label{eq:lem:app2-bis}
    \begin{split}
        &\displaystyle\sum_{a_1+\cdots +a_k=n,\atop \forall a_i\in \N^*}(-1)^{n-k}\dfrac{\ \ (n-k)!}{\ \ n!n^{n-k}a_1\cdots a_k}\dbinom{n-r}{k-r}a_1^{b_1}\cdots a_r^{b_r}\\
        =& \sum_{s=r}^{n-k+r}\displaystyle\sum_{a_1+\cdots +a_r=s, a_{r+1}+\cdots +a_k=n-s,\atop \forall a_i\in \N^*} (-1)^{n-k}\dfrac{\ \ (n-k)!}{\ \ n!n^{n-k}a_1\cdots a_k}\dbinom{n-r}{k-r}a_1^{b_1}\cdots a_r^{b_r}\\
        =& \dfrac{(-1)^{n-r}(n-r)!}{n!n^{n-r}}\sum_{s=r}^{n-k+r}\displaystyle\sum_{a_1+\cdots +a_r=s, a_{r+1}+\cdots +a_k=n-s,\atop \forall a_i\in \N^*}a_1^{b_1-1}\cdots a_r^{b_r-1} \dfrac{(-n)^{k-r}}{a_{r+1}\cdots a_k (k-r)!}\\
        =& \dfrac{(-1)^{n-r}(n-r)!}{n!n^{n-r}}\sum_{s=r}^{n-k+r}\displaystyle\sum_{a_1+\cdots +a_r=s,\atop \forall a_i\in \N^*}a_1^{b_1-1}\cdots a_r^{b_r-1}\left(\displaystyle\sum_{a_{r+1}+\cdots +a_k=n-s,\atop \forall a_i\in \N^*}\dfrac{(-n)^{k-r}}{a_{r+1}\cdots a_k (k-r)!}\right).
    \end{split}
\end{equation}
It follows that
\begin{equation*}
    \begin{split}
        &\sum_{k=r}^n\displaystyle\sum_{a_1+\cdots +a_k=n,\atop \forall a_i\in \N^*}(-1)^{n-k}\dfrac{(n-k)!}{n!n^{n-k}a_1\cdots a_k}\dbinom{n-r}{k-r}a_1^{b_1}\cdots a_r^{b_r}\\
        \overset{\eqref{eq:lem:app2-bis}}{=}& \dfrac{(-1)^{n-r}(n-r)!}{n!n^{n-r}}\sum_{k=r}^n\sum_{s=r}^{n-k+r}\displaystyle\sum_{a_1+\cdots +a_r=s,\atop \forall a_i\in \N^*}a_1^{b_1-1}\cdots a_r^{b_r-1}\left(\displaystyle\sum_{a_{r+1}+\cdots +a_k=n-s,\atop \forall a_i\in \N^*}\dfrac{(-n)^{k-r}}{a_{r+1}\cdots a_k (k-r)!}\right)\\
         =& \dfrac{(-1)^{n-r}(n-r)!}{n!n^{n-r}}\sum_{k=0}^{n-r}\sum_{s=r}^{n-k}\displaystyle\sum_{a_1+\cdots +a_r=s,\atop \forall a_i\in \N^*}a_1^{b_1-1}\cdots a_r^{b_r-1}\left(\displaystyle\sum_{a_{r+1}+\cdots +a_{k+r}=n-s,\atop \forall a_i\in \N^*}\dfrac{(-n)^k}{a_{r+1}\cdots a_{k+r} k!}\right)\\
         =& \dfrac{(-1)^{n-r}(n-r)!}{n!n^{n-r}}\sum_{s=r}^n\sum_{k=0}^{n-s}\displaystyle\sum_{a_1+\cdots +a_r=s,\atop \forall a_i\in \N^*}a_1^{b_1-1}\cdots a_r^{b_r-1}\left(\displaystyle\sum_{a_{r+1}+\cdots +a_{k+r}=n-s,\atop \forall a_i\in \N^*}\dfrac{(-n)^k}{a_{r+1}\cdots a_{k+r} k!}\right)\\
          =& \dfrac{(-1)^{n-r}(n-r)!}{n!n^{n-r}}\sum_{s=r}^n\displaystyle\sum_{a_1+\cdots +a_r=s,\atop \forall a_i\in \N^*}a_1^{b_1-1}\cdots a_r^{b_r-1}\left(\sum_{k=0}^{n-s}\displaystyle\sum_{a_{r+1}+\cdots +a_{k+r}=n-s,\atop \forall a_i\in \N^*}\dfrac{(-n)^k}{a_{r+1}\cdots a_{k+r} k!}\right)\\
          \overset{\eqref{eq:lem:app1}}{=}& \dfrac{(-1)^{n-r}(n-r)!}{n!n^{n-r}}\sum_{s=r}^n(-1)^{n-s}\dbinom{n}{n-s}\left(\displaystyle\sum_{a_1+\cdots a_r=s,\atop \forall a_i\in \N^*}a_1^{b_1-1}\cdots a_r^{b_r-1}\right)\\
          =&\dfrac{(-1)^r(n-r)!}{n!n^{n-r}}\sum_{s=r}^n(-1)^s\dbinom{n}{s}\left(\displaystyle\sum_{a_1+\cdots a_r=s,\atop \forall a_i\in \N^*}a_1^{b_1-1}\cdots a_r^{b_r-1}\right).
    \end{split}
\end{equation*}
\end{proof}

We will use the next notations in the rest of the paper: Let $b\in (\N^*)^r$ with $b_1+b_2+\cdots+b_r=n$, we define the set $\calA_b=\{c_j~|~c_j=\sum_{k=1}^j b_k,j\in[r]\}\subset [n]$ with $n\in \calA_b$; Let $S$ be a non-empty subset of $[n]$, we define the couple $(\calS,\theta)$ where $\theta=(\theta_1,\theta_2,\cdots,\theta_n)\in \calS^n$ such that for each $c_j\in \calA_b, \theta_{c_j}\ge \theta_k, \forall k<c_j$ and $\theta_{c_j}<\theta_k, \forall k>c_j$ and $\theta_n=\max\calS$; Let $\calC$ be the set of all couples $(\calS,\theta)$; Let $wt(\calS,\theta)$ be the weight of $(\calS,\theta)\in \calC$ defined by
$$wt(\calS,\theta)=(-1)^{|\calS|}.$$ 
We have the following lemmas to simplify the the right hand side of \eqref{eq:lem:app2}.

\begin{lem}\label{lem:2.3}
Let $n>r>0$ and $b\in (\N^*)^r$ with $b_1+b_2+\cdots+b_r=n$. Then  
\begin{equation}\label{eq:app1}
   \sum_{s=r}^n(-1)^s\dbinom{n}{s}\displaystyle\sum_{a_1+\cdots a_r=s,\atop \forall a_i\in \N^*}a_1^{b_1-1}\cdots 
a_r^{b_r-1}=\displaystyle\sum_{(S,\theta)\in\calC}wt(\calS,\theta).
\end{equation}
\end{lem}
\begin{proof}
$\rhd$ First, given a subset $\calS\subset [n]$, let us denote the subset $\calC_{\calS}=\{(\calS,\theta): \theta\in \calS^n\}$ of $\calC$. Then we will show that there is a correspondence between the number of elements of $P_{\calS}$
and $\sum_{a_1+\cdots a_r=s,\atop \forall a_i\in \N^*}a_1^{b_1-1}\cdots 
a_r^{b_r-1}.$

Consider the next procedure: suppose that $\calS=\{d_1,d_2,\cdots,d_s\}\subset[n]$ where $d_1<d_2<\cdots<d_s$. By dividing $\calS$ into $r$ parts in the order of $\calS$ as 
\begin{equation}\label{eq:div}
\underbrace{d_1,\cdots,d_{i_1}}_{a_1}\Big{|}\underbrace{d_{i_1+1},\cdots,d_{i_2}}_{a_2}\Big{|}\cdots\Big{|}\underbrace{d_{i_{r-1}+1},\cdots,d_{i_r}}_{a_r},
\end{equation}
where $a_i\in \N^* (i\in[r])$ denotes the number of elements of the $i$-th part and $d_{i_r}=d_s$. Now, let us consider $\theta_i, i\in [n]$ as the number of balls in the $i$th box, then there are two cases:\\ 
$\bullet$ For the values of $\theta_k$ with $k\in \mathcal{A}_b=\{c_1,\ldots, c_r\}$, then we have $\theta_{c_j}=d_{i_j}, \forall j\in [r]$, i.e., only one choice to put $d_{i_j}$ balls into the $c_j$-th ($c_j\in \mathcal{A}_b$) box as illustrated as follows:
$$\begin{tabular}{rcccccccccc}
&  &  & $d_{i_1}$&$\cdots$&$d_{i_2}$ &$\cdots$ & $d_{i_j}$ &$\cdots$ & $d_{i_r}$&\\
&  &  & $\downarrow$& & $\downarrow$ & & $\downarrow$ & & $\downarrow$& \\
& ( $\tiny\boxed{\phantom{O}}$ &$\cdots$&$\tiny\boxed{\phantom{O}}$&$\cdots$ & $\tiny\boxed{\phantom{O}}$ &$\cdots$ & $\tiny\boxed{\phantom{O}}$ &$\cdots$ & $\tiny\boxed{\phantom{O}}$&) \\[-2pt]
& ~~1& & $c_1$ &$\cdots$& $c_2$&$\cdots$ & $c_j$ &$\cdots$ & $c_r$& 
\end{tabular}$$
$\bullet$ For the values of $\theta_k$ with $k\notin \mathcal{A}_b$, we have: 
 \begin{enumerate}
\item[(i)] If $j=1$, then there are $c_1-1=b_1-1$ boxes on the left of the $c_1$-th box, by the definition of $\theta_k (k < c_1)$, each box $k$ can only have  $d_1,\cdots,d_{i_1}$ balls (i.e., $a_1$ possible choices), this implies that there are totally $a_1^{b_1-1}$ ways to fill the boxes with $k<c_1$.
\item[(ii)] Similarly, if $2\le j\le r$, then there are $c_{j}-c_{j-1}-1=b_{j}-1$ boxes between
 the $c_{j-1}$-th box and the $c_{j}$-th boxe. By the definition of $\theta_k (c_{j-1} < k < c_j)$, each box can only have $d_{i_{j-1}+1},\cdots,d_{i_j}$ balls (i.e., $a_j$ possible choices), this implies that there are totally $a_j^{b_j-1}$ ways to fill the boxes with $c_{j-1} < k < c_j$.
\end{enumerate}
Therefore, for each division \eqref{eq:div}, there are totally $a_1^{b_1-1}\cdots
a_r^{b_r-1}$ choices; then for all divisions on the set $\calS$, we have totally $\sum_{a_1+\cdots a_r=s, a_i\in \N^*}a_1^{b_1-1}\cdots 
a_r^{b_r-1}$ choices. Clearly, for any given $\calS\subset [n]$ arranged in increasing order, all possible $\theta$ constructed using the above procedure consists of a couple $(\calS,\theta) \in \calC_\calS$; conversely, any couple $(\calS,\theta)\in \calC_\calS$ can be interpreted as in above procedure. Hence, we have 
$$|\calC_\calS|=\sum_{a_1+\cdots a_r=s,\atop \forall a_i\in \N^*}a_1^{b_1-1}\cdots 
a_r^{b_r-1}.$$
$\rhd$ Now, consider all subsets $\calS$ of $[n]$, we have 
\begin{equation*}
\begin{split}
    \sum_{s=r}^n(-1)^s\dbinom{n}{s}\displaystyle\sum_{a_1+\cdots a_r=s,\atop \forall a_i\in \N^*}a_1^{b_1-1}\cdots 
a_r^{b_r-1}=&\displaystyle\sum_{\calS\subset [n],|\calS|\ge r}(-1)^{|\calS|}|\calC_\calS|\\
 =&\displaystyle\sum_{\calS\subset [n],|\calS|\ge r}\sum_{(\calS,\theta)\in\calC_\calS}wt(\calS,\theta)\\
 =&\displaystyle\sum_{(\calS,\theta)\in\calC}wt(\calS,\theta). 
\end{split}
\end{equation*}
\end{proof}
\begin{defn}[Involution, see e.g., \cite{stanley2011enumerative,loehr2011bijective}]
 An \emph{involution} on a set $X$ is a function $\tau$ such that $\tau$ is a bijection on $X$ and $\tau=\tau^{-1}.$
\end{defn}
\begin{lem}\label{lem:app3}
There is an involution $\phi$ on $\calC$ such that $wt(\calS,\theta)=-wt(\phi(\calS,\theta)).$
\end{lem}
\begin{proof}
For any $(\calS,\theta)\in \calC$, let $\Phi(\calS,\theta)\defas\{\theta_k\ |\ k=1,2,\cdots,n\}$, then $\Phi(\calS,\theta)\subset \calS$ and $\max \calS\in \Phi(\calS,\theta)$, by the definition of $(\calS,\theta)$. Let $\calU=\calS\setminus \Phi(\calS,\theta)$ and $\calV=\{w\ |\ w\in [n]\setminus \calS, w<\max \calS\}.$\\
$\rhd$ If $\calU\cup \calV \ne \emptyset$ and let $u=\max\{ \calU\cup \calV\} $.
\begin{itemize}
    \item[(i)] If $u\in \calS$, then $u\in \calU$ and $u\notin \Phi(\calS,\theta)$, thus we can well define $$\phi(\calS,\theta)\defas(\calS\setminus \{u\},\theta);$$ \item[(ii)] If $u\notin \calS$, then we can define $$\phi(\calS,\theta)\defas (\calS\cup\{u\},\theta).$$ 
\end{itemize}
$\rhd$ If $\calU\cup \calV = \emptyset$, then $\calS = \Phi(\calS,\theta)=\{1,2,\cdots,|\calS|\}$. Let $\calW$ be the set of positive integers appearing at least twice in $\theta$; Let $e_n=(0,.\cdots0,1)\in \R^n$. Consider $\calW$ in the following two cases:
\begin{itemize}
    \item If $\calW\ne\emptyset$ and let $w=\max \calW.$ 
    \begin{itemize}
        \item[(i)] If $w=\max \calS$, then $w<n$ and we can define $$\phi(\calS,\theta)\defas(\calS\cup\{1+\max \calS\},\theta+e_n);$$ \item[(ii)] if $w\ne \max \calS$, then we can define $$\phi(\calS,\theta)\defas(\calS\setminus\{\max \calS\},\theta-e_n).$$ 
    \end{itemize}
    \item If $\calW=\emptyset$, then we must have $|\calS|=n,$ and we can define $$\phi(\calS,\theta)\defas(\calS\setminus \{n\},\theta-e_n).$$
\end{itemize}
 It is easy to see that $\phi$ is an involution with $wt(\calS,\theta)=-wt(\phi(\calS,\theta))$ for any $(\calS,\theta)\in \calC$.
\end{proof}

\begin{thm}\label{th:1}
Let $1\leq r \leq n$ and $b=(b_1,\cdots,b_r)^T\in (\N^*)^r.$ Then 
\begin{equation}\label{eq:11}
   \dbinom{n}{b_1,\cdots,b_r}\sum_{k=r}^{n}\displaystyle\sum_{a_1+\cdots a_k=n,\atop \forall a_i\in \N^*}(-1)^{n-k}\dfrac{\ \ (n-k)!}{\ \ n!n^{(n-k)}a_1\cdots a_k}\dbinom{n-r}{k-r}a_1^{b_1}\cdots a_r^{b_r}=\delta_{nr},
\end{equation}
where $$\delta_{nr}\defas \begin{cases}
1 , & n = r\\ 
0 , & n \ne r
\end{cases}.$$
\end{thm}
\begin{proof}
$\rhd$ If $n=r$, it implies that $a_i=b_i=1(i\in[n])$. Then we get 
\begin{equation*}
    \dbinom{n}{b_1,\cdots,b_r}\sum_{k=r}^{n}\displaystyle\sum_{a_1+\cdots a_k=n,\atop \forall a_i\in \N^*}(-1)^{n-k}\dfrac{\ \ (n-k)!}{\ \ n!n^{(n-k)}a_1\cdots a_k}\dbinom{n-r}{k-r}a_1^{b_1}\cdots a_r^{b_r}=1.
\end{equation*}
$\rhd$ If $n \ne r$, then combining identities \eqref{eq:lem:app2} and \eqref{eq:app1}, we have 
$$\sum_{k=r}^{n}\displaystyle\sum_{a_1+\cdots a_k=n,\atop \forall a_i\in \N^*}(-1)^{n-k}\dfrac{\ \ (n-k)!}{\ \ n!n^{(n-k)}a_1\cdots a_k}\dbinom{n-r}{k-r}a_1^{b_1}\cdots a_r^{b_r}=\displaystyle\sum_{(\calS,\theta)\in\calC}wt(\calS,\theta).$$
It follows from Lemma \ref{lem:app3} that  $\forall(\calS,\theta)\in \calC$, we have $\phi(\calS,\theta)\in \calC$, thus $\displaystyle\sum_{(\calS,\theta)\in\calC}wt(\calS,\theta)=0,$ then

$$\dbinom{n}{b_1,\cdots,b_r}\sum_{k=r}^{n}\displaystyle\sum_{a_1+\cdots a_k=n,\atop \forall a_i\in \N^*}(-1)^{n-k}\dfrac{\ \ (n-k)!}{\ \ n!n^{(n-k)}a_1\cdots a_k}\dbinom{n-r}{k-r}a_1^{b_1}\cdots a_r^{b_r}=0. $$
\end{proof}
\begin{thm}\label{th:2}
Let $x\in \R^n$, then we have 
\begin{equation}\label{eq:9}
\prod_{i=1}^n x_i = 
\sum_{\alpha\in B(n,n)}(-1)^{n-\|\alpha\|_0}\frac{(n-\|\alpha\|_0)!}{\ \ n!n^{(n-\|\alpha\|_0)}\displaystyle\prod\limits_{\alpha_i\ne 0}\alpha_i}\ \ \langle \alpha, x\rangle^n,
 \end{equation}
\end{thm}
where $\|\cdot\|_0$ is the zero norm. \footnote{Let $x\in \R^n$, $\|x\|_0$ is the number of nonzero elements of $x$.} 
\begin{proof}
Denote the right hand side of \eqref{eq:9} as
\begin{equation*}
f(x)=\sum_{\alpha\in B(n,n)}(-1)^{n-\|\alpha\|_0}\frac{(n-\|\alpha\|_0)!}{\ \ n!n^{(n-\|\alpha\|_0)}\displaystyle\prod\limits_{\alpha_i\ne 0}\alpha_i}\ \ \langle \alpha,x\rangle^n.
\end{equation*}
Due to the fact that
$\langle \alpha, x \rangle^n=\displaystyle\sum_{\beta\in B(n,n)} \tbinom{n}{\beta} {\alpha}^{\beta} x^{\beta},$
then
\begin{equation*}
    \begin{split}
f(x)=&\sum_{\alpha\in B(n,n)}(-1)^{n-\|\alpha\|_0} \frac{(n-\|\alpha\|_0)!}{\ \ n!n^{(n-\|\alpha\|_0)} \displaystyle \prod \limits_{\alpha_i\ne 0} \alpha_i} \Big(\sum_{\beta\in B(n,n)} \dbinom{n}{\beta} \alpha^{\beta} x^{\beta}\Big)\\
    =&\sum_{\alpha\in B(n,n)}\sum_{\beta\in B(n,n)}(-1)^{n-\|\alpha\|_0}\dbinom{n}{\beta}\frac{(n-\|\alpha\|_0)!\ \alpha^{\beta}}{\ \ n!n^{(n-\|\alpha\|_0)} \displaystyle\prod\limits_{\alpha_i\ne 0}\alpha_i} x^{\beta}\\
    =&\displaystyle\sum_{\beta\in B(n,n)}\dbinom{n}{\beta}\sum_{\alpha\in B(n,n)}(-1)^{n-\|\alpha\|_0}\frac{(n-\|\alpha\|_0)!\ \alpha^{\beta}}{\ \ n!n^{(n-\|\alpha\|_0)} \displaystyle\prod\limits_{\alpha_i\ne 0}\alpha_i} x^{\beta}.       \end{split}
 \end{equation*}
Denote $c_{\beta}(\beta\in B(n,n))$ by 
\begin{equation*}
    c_{\beta}=\sum_{\alpha\in B(n,n)}(-1)^{n-\|\alpha\|_0}\dbinom{n}{\beta}\frac{(n-\|\alpha\|_0)!\ \alpha^{\beta}}{\ \ n!n^{(n-\|\alpha\|_0)} \displaystyle\prod\limits_{\alpha_i\ne 0}\alpha_i}.
\end{equation*}
Without loss of generality, let $\tilde{\alpha}=\beta$ and we assume that the first r-elements of $\tilde{\alpha}$ are nonzeros $$\tilde{\alpha}=(b_1,\cdots,b_r,0,\cdots,0),$$
and $b_1+b_2+\cdots+b_r=n, b_i\in \N^*\ (i\in [r])$. Then we have 
\begin{equation*}
\begin{split}
  c_{\tilde{\alpha}}&=\sum_{k=r}^{n}(-1)^{n-k}\dbinom{n}{b_1,\cdots,b_r}\dbinom{n-r}{k-r}\dfrac{\ \ (n-k)!}{\ \ n!n^{(n-k)}}\displaystyle\sum_{a_1+\cdots a_k=n,\atop \forall a_i\in \N^*}\dfrac{a_1^{b_1}\cdots a_r^{b_r}}{a_1\cdots a_k}\\
  &=\dbinom{n}{b_1,\cdots,b_r}\sum_{k=r}^{n}\displaystyle\sum_{a_1+\cdots a_k=n,\atop \forall a_i\in \N^*}(-1)^{n-k}\dfrac{\ \ (n-k)!}{\ \ n!n^{(n-k)}a_1\cdots a_k}\dbinom{n-r}{k-r}a_1^{b_1}\cdots a_r^{b_r}.\\
  &\overset{\eqref{eq:11}}{=}\delta_{n,r}.
\end{split}
\end{equation*}
\end{proof}
 Next, we will use basic linear algebra theory in matrix analysis (see e.g., \cite{horn2012matrix}) to a homogeneous polynomial space $\bbH_d[x]$ to prove Theorem \ref{th:3}.
\begin{thm}\label{th:3}
The matrices $\whV(n,d)$ and $V(n,d)$ are nonsingular.
\end{thm}
\begin{proof}
For any monomial $x^{\alpha}$ with $x\in\R^n$ and $\alpha\in B(n,d)$, we can present \begin{equation}\label{eq:alpha}
 x^{\alpha}=x_{i_1}^{r_1}x_{i_2}^{r_2}\cdots x_{i_k}^{r_k}=\underbrace{x_{i_1}\times\cdots \times x_{i_1}}_{r_1}\times \underbrace{x_{i_2} \times\cdots \times x_{i_2}}_{r_2}\cdots\times \underbrace{x_{i_k}\times \cdots \times x_{i_k}}_{r_k},
\end{equation}
where $|\alpha|=\sum_{i=1}^{k}r_i, (r_1,\cdots,r_k)\in (\N^*)^k $ with $k\le \min\{n,d\}$ and $i_1<i_2<\cdots<i_k.$
We get from Theorem \ref{th:2} that
\begin{equation}\label{eq:13}
    \prod_{i=1}^dy_i=\sum_{\beta\in B(d,d)} c_{\beta}\langle \beta, y\rangle^d,
\end{equation}
where $y\in\R^d$
and $c_{\beta}=(-1)^{d-\|\beta\|_0}(d-\|\beta\|_0)!\big/( d!d^{(d-\|\beta\|_0)}\prod\limits_{\beta_i\ne 0}\beta_i).$
Let us partition $[d]$ into $k$ sets $I_j(j\in[k])$ such that $|I_j|=r_j$. Then, we can rewrite $\langle \beta, y\rangle^d$ as 
\begin{equation}\label{eq:beta}
\langle \beta, y\rangle^d=\bigg(\sum_{j=1}^k\sum_{i\in I_j}\beta_iy_i\bigg)^d.    
\end{equation}
By taking $y=(\underbrace{x_{i_1},\cdots,x_{i_1}}_{r_1},\cdots,\underbrace{x_{i_k},\cdots,x_{i_k}}_{r_k})$,
it follows from \eqref{eq:alpha}, \eqref{eq:13} and \eqref{eq:beta} that 
\begin{equation}\label{eq:calpha}
  x^{\alpha}=\sum_{\beta\in B(d,d)}c_{\beta}\bigg(\sum_{j=1}^k\big(\sum_{i\in I_j}\beta_i\big)x_{i_j}\bigg)^d.
\end{equation}
For each $\beta \in B(d,d)$ and by denoting $u_j=\sum_{i\in I_j}\beta_i$, we get from $u_j\in\N^*$ and $\sum_{j=1}^{k}u_j=d$ that there exists a unique $\gamma\in B(n,d)$ in form of : 
$$\begin{tabular}{crcccccccccccccc}
&$\gamma=$& ( $0$ &$\cdots$&$u_1$&$\cdots$&0 & $\cdots$&$u_j$ &$\cdots$&0 & $\cdots$&$u_k$ &$\cdots$ & $0$&) \\[-2pt]
& &~~1&$\cdots$& $i_1$ &$\cdots$&$\cdot$& $\cdots$&$i_j$&$\cdots$&$\cdot$& $\cdots$&$i_k$ &$\cdots$ & $n$&
\end{tabular}$$
i.e., $\gamma$ is a vector of $\N^n$ such that $\gamma_{i_j} = u_j, \forall j\in [k]$ and $0$ for others, satisfying 
\begin{equation}\label{eq:gamma}
 \bigg(\sum_{j=1}^ku_jx_{i_j}\bigg)^d = \langle \gamma, x \rangle^d. 
\end{equation}
Therefore, it follows from \eqref{eq:calpha} and \eqref{eq:gamma} that  $\forall\alpha\in B(n,d),$
$$x^{\alpha}\in\Span\{\langle\gamma, x\rangle^n : \gamma\in B(n,d)\}.$$
Then, the identity \eqref{eq:iden} yields that $\whV(n,d)$ is nonsingular. Hence, the matrix $V(n,d)$ is also nonsingular based on \eqref{eq:relationVvsVh}.\\
\end{proof}
\begin{cor}
The set $\mathcal{B}=\{\langle \alpha,x\rangle^d:\alpha\in B(n,d)\}$ is a (polynomial) basis of $\bbH_d[x]$.
\end{cor}
\begin{proof}
This is an immediate consequence of Theorem \eqref{th:3} and identity \eqref{eq:iden}.
\end{proof}
Next lemma tells us that the matrix $V(n,d)$ has a block lower triangular form by some permutation of elements in $B(n, d)$. To see this, we define the following  lexicographical order.
\begin{defn}[Logically Reverse Lexicographical Order]
Given the set $\N$, and two sequences of numbers from $\N$ of length $n>0$, say $a=(a_1,a_2,\cdots,a_n)$ and $b=(b_1,b_2,\cdots,b_n)$. We call that $a$ is smaller than $b$ in logically reverse lexicographical order, denoted by
$$a \lrlex b$$ if one of the following conditions is verified: 
\begin{enumerate}
    \item[(i)] If $\Log(a)\neq \Log(b)$ and $\Log(a_i)>\Log(b_i)$ for the first $i$ where $a_i\neq b_i$;
    \item[(ii)] If $\Log(a)= \Log(b)$ and $a_j>b_j$ for the first $j$ where $a_j\neq b_j$.
\end{enumerate}
\end{defn}

\begin{exmp}
Given sequences $a=(1,2,0,3), b=(2,1,0,3), c=(1,2,3,0)$, Then $ \Log(a)=(1,1,0,1), \Log(b)=(1,1,0,1), \Log(c)=(1,1,1,0)$, and we have 
$$c\lrlex b\lrlex a.$$
\end{exmp}
\begin{lem}\label{lem:app4}
 Based on the logically reverse lexicographical order, the matrix $V(n,d)$ can be presented as a block lower triangular form as: 
 \begin{equation}\label{eq:blockV(n,d)}
 V(n,d)= \left[
 \begin{array}{ccccccc}
 \wtV_{(1,1)} & & & & & &\\
 \vdots& \ddots& & & & &\\
 * & \cdots & \wtV_{(1,n_1)}& & & &\\
 \vdots & \vdots & \vdots&\ddots & & &\\
 * & * & * & \cdots &\wtV_{(p,1)} &  &\\
 \vdots & \vdots & \vdots &\vdots &\vdots&\ddots &\\
 * & * & * & * & *& \cdots & \wtV_{(p,n_p)}
\end{array}
\right],
\end{equation}
where $p=\min\{n,d\}$. Moreover, \\
(a) $n_k=\tbinom{n}{k},\ \forall k\in [p],$ and
$$\wtV_{(k,1)}=\cdots=\wtV_{(k,n_k)},$$
where the block matrix $\wtV_{(k,t)}(\forall t\in [n_k])$ is exactly the power-product matrix generated by all k-compositions of integer $d$, i.e., 
\[
\wtV_{(k,t)}=
\left[
\begin{array}{ccc}
  {(\gamma^1)}^{\gamma^1}&\cdots& {(\gamma^1)}^{\gamma^{q}}\\
  \vdots&\ddots&\vdots\\
  {(\gamma^q)}^{\gamma^1}&\cdots&{(\gamma^{q})}^{\gamma^{q}}
 \end{array}
   \right]
\]
 with $|\gamma^i|=d, \gamma^i\in (\N^*)^k\ (i\in [q], q=\tbinom{d-1}{k-1}).$\\
(b) The matrix $\whV(n,d)$ is also a block lower triangular matrix.
\end{lem}

\begin{proof}
Let $\calS^k=\{\alpha: \|\alpha\|_0=k, \alpha \in B(n,d)\}$ where $k\in [p]$, then $B(n,d)$ is partitioned as: 
$$B(n,d)=\cup_{k=1}^{p}\calS^k,\ \calS^i\cap \calS^j = \emptyset,\ i\ne j, \ \forall i,j\in[p].$$
Sort $\calS^k$ by \emph{logically reverse lexicographical order} as $\calS^k_{\lrlex}$ and sort $B(n,d)$ as $\{\calS^1_{\lrlex},\cdots, \calS^p_{\lrlex}\}.$
Due to the fact that ${(\alpha^i)}^{\alpha^j}=0$ if and only if $\exists\omega_0\in[n]$ such that $\alpha^i_{\omega_0}=0, \alpha^j_{\omega_0} \ne 0.$ Therefore,\\ 
(i) Denote $n_k=|\calS^k_{\lrlex}|$, clearly, we have
$$n_k=\tbinom{n}{k}.$$ Then $\forall \alpha^i,\alpha^j\in \calS^k_{\lrlex}(i,j\in[n_k])$, the block $\wtV_k=[(\alpha^i)^{\alpha^j}]_{(i,j)}$ has the following form
 \[
 \wtV_k=
 \left(
 \begin{array}{ccc}
  (\alpha^1)^{\alpha^1}&\cdots& (\alpha^1)^{\alpha_{n_k}}\\
  \vdots&\ddots&\vdots\\
  (\alpha^{n_k})^{\alpha_1}&\cdots&(\alpha^{n_k})^{\alpha^{n_k}}
 \end{array}
   \right)
   =
    \left(
 \begin{array}{ccc}
  \wtV_{(k,1)}&\cdots& 0\\
  \vdots&\ddots&\vdots\\
  0&\cdots&\wtV_{(k,n_k)} 
 \end{array}
   \right),\]
   where the block matrix $\wtV_{(k,t)}$ is exactly the power-product matrix generated by all k-compositions of integer $d$, of which  the notation is referred to, see e.g., \cite{heubach2004compositions,dickson1929introduction}:
    \[
\wtV_{(k,t)}=
\left[
\begin{array}{ccc}
  {(\gamma^1)}^{\gamma^1}&\cdots& {(\gamma^1)}^{\gamma^{q}}\\
  \vdots&\ddots&\vdots\\
  {(\gamma^q)}^{\gamma^1}&\cdots&{(\gamma^{q})}^{\gamma^{q}}
 \end{array}
   \right]
\]
with $|\gamma^i|=d, \gamma^i\in (\N^*)^k\ (i\in [q], q=\tbinom{d-1}{k-1}).$ Therefore, $$\wtV_{(k,1)}=\cdots=\wtV_{(k,n_k)}.$$ 
 (ii) $\forall l>k$, $\forall \alpha^i\in \calS^k_{\lrlex}, \forall \alpha^j\in \calS^l_{\lrlex}$, we have $(\alpha^i)^{\alpha^j}=0.$\\
We conclude from (i) and (ii) that the matrix $V(n,d)$ have a block lower triangular form as \eqref{eq:blockV(n,d)}. Consequently, by Equation \eqref{eq:relationVvsVh}, the matrix $\whV(n,d)$ is also in block lower triangular form.
\end{proof}
\section{Sparsity of $V(n,d)$}
For any positive integers $n$ and $d$, the size of the matrix $V(n,d)$ (or $\whV(n,d)$) is $\binom{n+d-1}{d}^2$. This could be a very large matrix, e.g., the size of $V(n,d)$ with $(n,d)=(100,4)$ reaches up to millions! In this section, we are interested in the sparsity of the matrix $V(n,d)$ with respect to $n$ and $d$. Let us denote $\nnz(M)$ by the number of nonzero entries of a matrix $M\in \R^{n\times m}$, and $\spa(M)$ the sparsity of $M$ defined by 
$$\spa(M) = 1 - \frac{\nnz(M)}{n\times m},$$
then we have the sparsity theory as follows:
\begin{thm}[Sparsity theory]\label{th:4}
$$\nnz(V(n,d))=\sum_{k=1}^{\min\{n,d\}}\tbinom{n}{k}\tbinom{d-1}{k-1}\tbinom{d+k-1}{d};\quad \spa(V(n,d))=1-\frac{\nnz(V(n,d))}{\tbinom{n+d-1}{n}^2}.$$ 
\end{thm}
\begin{proof}
Let us denote the row index  of the $l$-th row of the diagonal block $\wtV(k,t)$ in $V(n,d)$ by $R_{ktl}$. The $R_{ktl}$-th row of $V(n,d)$ consists of elements $(\alpha^i)^{\alpha^j}$ with $i=R_{ktl}, j\in[s_{n,d}])$ and $\nnz(\alpha^i)=k$. Based on the fact that
$$(\alpha^i)^{\alpha^j}\ne 0\ \Leftrightarrow \forall \omega\in[n], \alpha^i_{\omega}=0 \Rightarrow \alpha^j_{\omega}=0,$$
clearly, the number of nonzero elements of the $R_{ktl}$-th row of $V(n,d)$ is exactly $\tbinom{d+k-1}{d}.$ We get from Lemma \ref{lem:app4} that each block $\wtV(k,t)$ has $\tbinom{d-1}{k-1}$ rows and there are $n_k=\tbinom{n}{k}$ equal blocks of $\wtV(k,t)$, then for all $k\in [\min\{n,d\}]$, we have $$\nnz(V(n,d))=\sum_{k=1}^{\min\{n,d\}} \tbinom{n}{k}\tbinom{d-1}{k-1}\tbinom{d+k-1}{d}.$$
It follows immediately from the definition of $\spa(V(n,d))$ that 
$$\spa(V(n,d))=1-\frac{\nnz(V(n,d))}{\tbinom{n+d-1}{n}^2}.$$ 
\end{proof}
Here are some discussions about the sparsity of the matrix $V(n,d)$ with respect to $n$ and $d$. Based on Theorem \ref{th:4}, it is easy to verify that  
\begin{itemize}
    \item[(i)] For fixed $n$, $\lim\limits_{d\to \infty}\spa(V(n,d))=0.$
    \item[(ii)] For fixed $d$, $\lim\limits_{n\rightarrow \infty}\spa(V(n,d))=1.$
\end{itemize}
which demonstrates that when $n$ is fixed, as $d\to\infty$, the matrix $V(n,d)$ is almost completely dense. However, according to Figure \ref{fig:m-spa}, when $d\leq20$, except for the matrices with fixed $n\leq 10$, the other matrices with larger $n$ have more than $90\%$ sparsity. Moreover, the \emph{sparsity} of the matrix $V(n,d)$ decreases slowly with respect to the increase of $d$ (with fat tail).
\begin{figure}[ht!]
    \centering
    \subfigure[$d$ v.s. sparsity] {\includegraphics[width=0.47\textwidth]{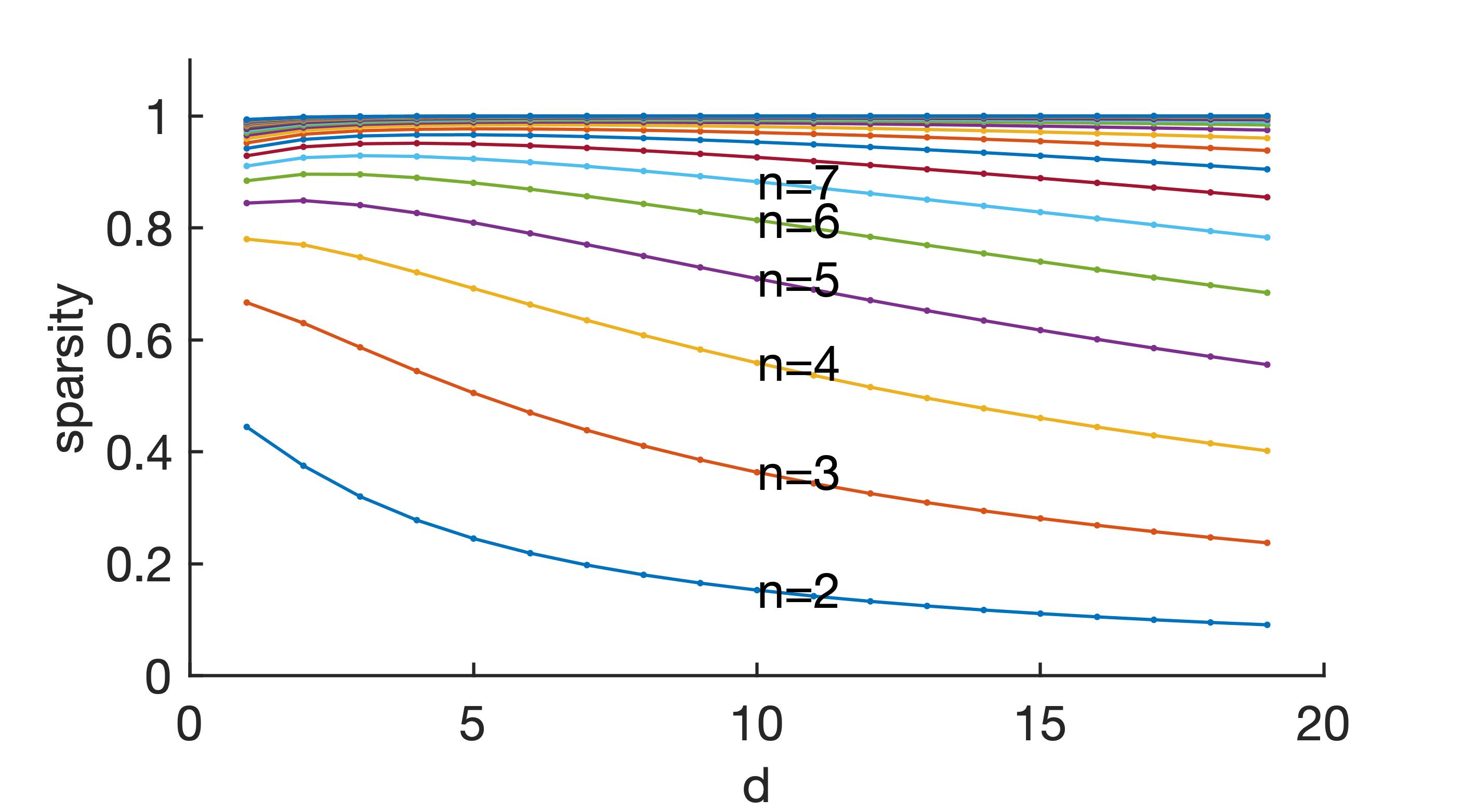}
    \label{fig:m-spa}}
    \subfigure[$n$ v.s. sparsity]
    {\includegraphics[width=0.47\textwidth]{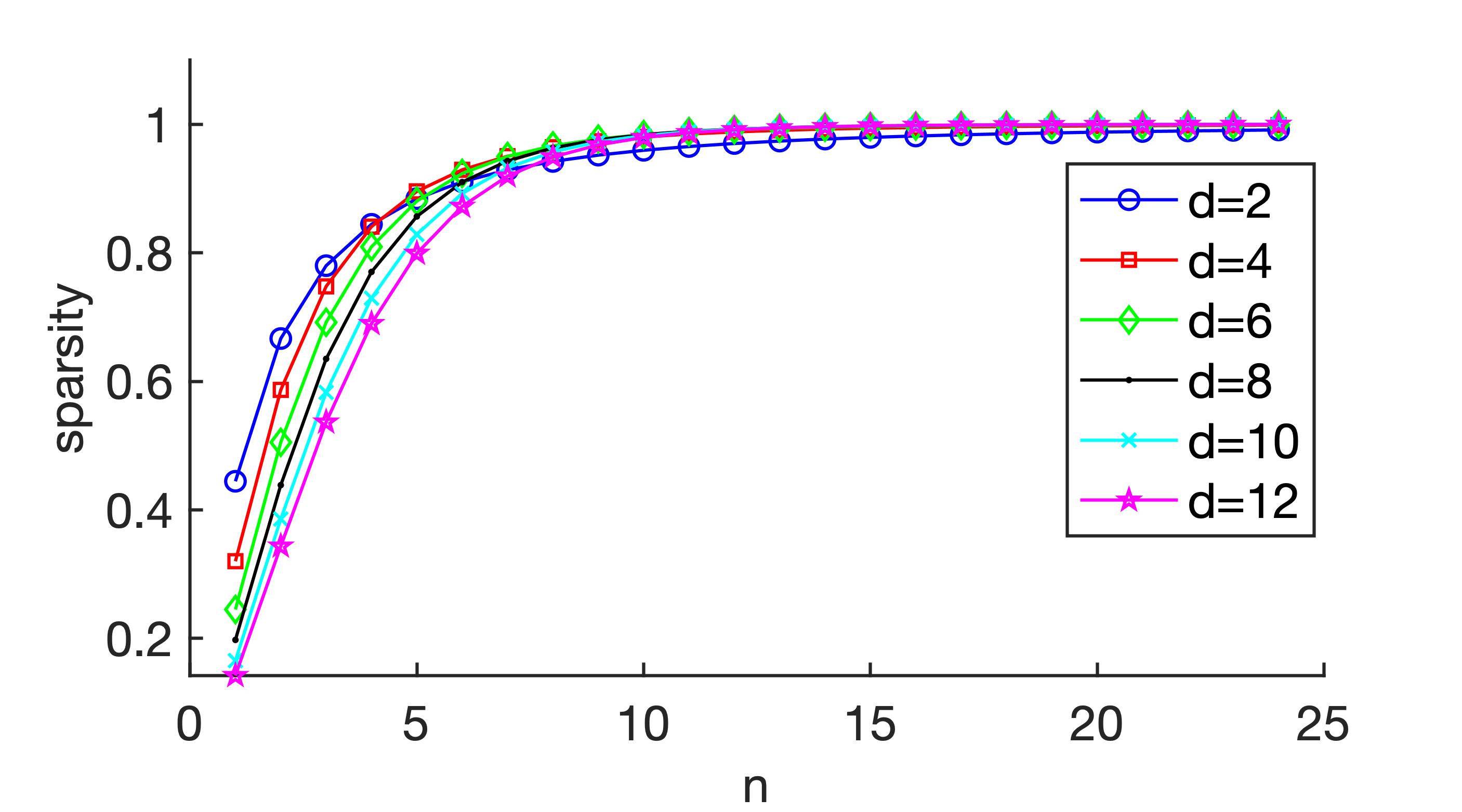}
    \label{fig:n-spa}}
    \caption{The relationship between $n$ or $d$ and sparsity}
\end{figure}
On the other hand, if $d$ is fixed, then $V(n,d)$ is almost a zero matrix as $n\to\infty$. We observe in Figure \ref{fig:n-spa} that for all fixed $d\in \{2,4,6,8,10,12\}$, the matrix $V(n,d)$ is always sparse with more than $90\%$ sparsity for $n\ge10$. Moreover, the sparsity of the matrix $V(n,d)$ seems to be more sensitive to $d$ than to $n$.
\begin{thm}[Inverse sparsity theory]
$$\spa(V^{-1}(n,d))\ge \spa(V(n,d)).$$
\end{thm}
\begin{proof}
The conclusion follows by the fact that $(V(n,d))_{ij}=0 \Rightarrow (V^{-1}(n,d))_{ij}=0$. To prove this, first, we know that $(V(n,d))_{ij}=0 \Rightarrow (V^{-1}(n,d))_{ij}=0$ if and only if $(\whV(n,d))_{ij}=0 \Rightarrow ((\whV(n,d))^{-1})_{ij}=0$, and we have the equivalence:
$$(\whV(n,d))_{ij}=0 \Leftrightarrow (\alpha^i)^{\alpha^j}=0 \Leftrightarrow  \text{there exists}\ \omega_0\in[n] \ \text{such that}\ \alpha^i_{\omega_0}=0\ \text{and} \ \alpha^j_{\omega_0} \ne 0.$$ 
 As we know that $x^{\alpha^i}$ is a linear combination of $\{\langle\beta, x\rangle^d : \beta\in B(n,d)\}=\{\langle\beta^1, x\rangle^d,\cdots,\langle\beta^{s_{n,d}}, x\rangle^d\}$ where all $\beta^l\in B(n,d)$, i.e.,
\[x^{\alpha^i} = \sum_{k=1}^{s_{n,d}} (\whV^{-1}(n,d ))_{ik}\langle \beta^k, x\rangle^d.\]
On the other hand, because of  $\alpha^i_{\omega_0}=0$, we can get from \eqref{eq:calpha} and \eqref{eq:gamma} that
$x^{\alpha^i}\in \Span\{\langle\gamma, x\rangle^d :\gamma_{\omega_0}=0, \gamma\in B(n,d)\}.$ Combining $\alpha^j_{\omega_0}\ne0$ and $\alpha^j\in B(n,d)$, we have $$\langle\alpha^j,x\rangle^d\notin\{\langle\gamma, x\rangle^d :\gamma_{\omega_0}=0, \gamma\in B(n,d)\}.$$
Since there is a unique representation of $x^{\alpha^i}$ in the basis $\{\langle\beta, x\rangle^d : \beta\in B(n,d)\}$, then $(\whV^{-1}(n,d))_{ij}=0$ and $(V^{-1}(n,d))_{ij}=0$. Therefore, the conclusion is verified immediately.  
\end{proof}

\section{Determinant of $V(n,d)$}
\begin{cor}\label{cor:4.1}
By representing the matrix $V(n,d)$ in block lower triangular form as in Lemma \ref{lem:app4}, then 
$$\det V(n,d) = \prod_{i=1}^{\min\{n,d\}} \left(\det \widetilde{V}_{(i,1)} \right)^{n\choose i}.$$
\end{cor}
\begin{proof}
Immediate consequence of Lemma \ref{lem:app4}.
\end{proof}

\begin{cor}
$$\det V(n,d) \neq 0.$$
\end{cor}
\begin{proof}
Immediate consequence of Theorem \ref{th:3}.
\end{proof}

\paragraph{Special case : compute $\det V(2,d)$ with $d\in \mathbb{N}^*$}
\begin{prop}\label{prop:1}
Let $A_k(a,b)$ be the $(k+1)\times (k+1)$ matrix 
as
$$A_k(a,b)=
\begin{bmatrix}a^k&a^{k-1}b&\cdots&ab^{k-1}&b^k\\
(a-1)^k&(a-1)^{k-1}(b+1)&\cdots&(a-1)(b+1)^{k-1}&(b+1)^k\\
\vdots&\vdots&&\vdots&\vdots\\
(a-k)^k&(a-k)^{k-1}(b+k)&\cdots&(a-k)(b+k)^{k-1}&(b+k)^k
\end{bmatrix}$$
whose $(i,j)$ element is $(a-i+1)^{k-j+1}(b+i-1)^{j-1}$, then we have $$\det A_k(a,b)=(a+b)^{k(k+1)/2}\prod_{h=1}^{k}h!.$$
\end{prop}

\paragraph{First proof for Proposition \ref{prop:1}}
\begin{proof}
By summing to each column the next one, we can factor $a+b$ from each of the first $k$ columns so that

$$\det A_k(a,b)=(a+b)^k 
\begin{vmatrix}a^{k-1}&a^{k-2}b&\cdots&b^{k-1}&b^k\\
(a-1)^{k-1}&(a-1)^{k-2}(b+1)&\cdots&(b+1)^{k-1}&(b+1)^k\\
\vdots&\vdots&&\vdots&\vdots\\
(a-k)^{k-1}&(a-k)^{k-2}(b+k)&\cdots&(b+k)^{k-1}&(b+k)^k
\end{vmatrix}.$$
Note that the exponents of the ``$a$" terms in the first $k$ columns are decreased by 1, in particular there are no ``$a$" terms in the last two columns. 

Reiterating, we can get rid of all ``$a$" terms:

$$\det A_k(a,b)=(a+b)^{k(k+1)/2} \begin{vmatrix}1&b&b^2&\cdots&b^{k-1}&b^k\\
1&(b+1)&(b+1)^2&\cdots&(b+1)^{k-1}&(b+1)^k\\
\vdots&\vdots&\vdots&&\vdots&\vdots\\
1&(b+k)&(b+k)^2&\cdots&(b+k)^{k-1}&(b+k)^k
\end{vmatrix}.$$
This last matrix is a Vandermonde matrix whose determinant is $$\prod_{0\leq j<i\leq k} (b+i)-(b+j) = \prod_{0\leq j<i\leq k} i-j = (k\times (k-1) \times \cdots 1) \times ((k-1) \times (k-2)\times 1) \times \cdots 1 = \prod_{h=1}^kh!.$$
It follows that 
$$\det A_k(a,b) = (a+b)^{k(k+1)/2}\prod_{h=1}^{k}h!.$$
\end{proof}

\paragraph{Second proof for Proposition \ref{prop:1}}
\begin{proof}
(i) Suppose that $a, a-1, \ldots, a-k$ are all non-zeros, then we can factorize $a^k$ from the first row of $\det A_k(a,b)$, and $(a-1)^k$ from the second row and so on, thus $$\det A_k(a,b) = a^k(a-1)^k\cdots (a-k)^k \begin{vmatrix}1 & \frac b a &\cdots & (\frac b a)^{k}\\
1 & (\frac{b+1}{a-1}) & \cdots & (\frac{b+1}{a-1})^{k}\\
\vdots&\vdots&\vdots&\vdots&\\
1 & (\frac{b+k}{a-k}) & \cdots & (\frac{b+k}{a-k})^{k}\\
\end{vmatrix}.$$
The last determinant is Vandermonde which is computed by:
\begin{align*}
    \prod_{0\leq j<i\leq k} \frac{b+i}{a-i}-\frac{b+j}{a-j} &= \prod_{0\leq j<i\leq k} \frac{\left( b+a\right) \,\left( i-j\right) }{\left( a-i\right) \,\left( a-j\right) } \\
    &= \prod_{0\leq j<i\leq k} \frac{\left( b+a\right) }{\left( a-i\right) \,\left( a-j\right) } \prod_{0\leq j<i\leq k} \left( i-j\right)  \\
    &=  (a+b)^{k(k+1)/2} \times \prod_{h=1}^kh!\times \prod_{0\leq j<i\leq k} \frac{1}{\left( a-i\right) \,\left( a-j\right) }\\
    &=  (a+b)^{k(k+1)/2} \times \prod_{h=1}^kh!\times \frac{1}{a^k(a-1)^k\cdots (a-k)^k}.
\end{align*}
It follows that 
$$\det A_k(a,b) = (a+b)^{k(k+1)/2}\prod_{h=1}^{k}h!.$$
(ii) Otherwise, suppose that there exists $i\in \{1,\ldots,k\}$ such that $a-i=0$, then  
$$
\det A_k(a,b)=(b+i)^k \begin{small}
\begin{vmatrix}a^k&a^{k-1}b&\cdots&ab^{k-1}\\
(a-1)^k&(a-1)^{k-1}(b+1)&\cdots&(a-1)(b+1)^{k-1}\\
\vdots&\vdots&&\vdots\\
(a-i+1)^k&(a-i+1)^{k-1}(b+1)&\cdots&(a-i+1)(b+1)^{k-1}\\
(a-i-1)^k&(a-i-1)^{k-1}(b+1)&\cdots&(a-i-1)(b+1)^{k-1}\\
\vdots&\vdots&&\vdots\\
(a-k)^k&(a-k)^{k-1}(b+k)&\cdots&(a-k)(b+k)^{k-1}
\end{vmatrix}\end{small}.$$
We can obtain in a similar way as in case (i) to deduce that $$\det A_k(a,b) = (a+b)^{k(k+1)/2}\prod_{h=1}^{k}h!.$$
\end{proof}
Now, we are ready to Compute $\det V(2,d)$ for $d\in \N^*$.
\begin{thm}
For any $d\in \mathbb{N}^*$, we have 
\begin{equation}
    \label{eq:detV(2,d)}
\det V(2,d) = d^{d(d+1)/2}\prod_{h=1}^dh!.
\end{equation}
\end{thm}
\begin{proof}
For $d=1$, we have $V(2,1) = \begin{bmatrix}
1 & 0 \\
0 & 1
\end{bmatrix}$, then $\det V(2,1) = 1,$ and the equation \eqref{eq:detV(2,d)} is verified.
For $d\geq 2$, we order the set $B(2,d)$ lexicographically as $\{(0,d),(1,d-1),...,(d,0)\}$ to get
$$V(2,d)=\begin{bmatrix}d^d&0&0&\cdots&0 & 0\\
*&(d-1)^{d-1}&(d-1)^{d-2}&\cdots&(d-1)&*\\
*&(d-2)^{d-1}2&(d-2)^{d-2}2^2&\cdots&(d-2)2^{d-1}&*\\
*&(d-3)^{d-1}3&(d-3)^{d-2}3^2&\cdots&(d-3)3^{d-1}&*\\
\vdots&\vdots&\vdots&&\vdots&\vdots\\
0&0&0&\cdots&0&d^d
\end{bmatrix}.$$
Note that the central matrix is $A_{d-2}(d-1,1)$ with the first row multiplied by $d-1$, the second multiplied by $(d-2)2$, the third multiplied by $(d-3)3$, and so on, then we get for $d\geq 2$ that 
$$\det V(2,d) =d^{2d}(d-1)!(d-1)!\det A_{d-2}(d-1,1)=d^{d(d+1)/2}\prod_{h=1}^dh!.$$
We conclude that for any $d\in \N^*$, the equation \eqref{eq:detV(2,d)} is verified.
\end{proof}

\begin{conj}\label{conj:1}
Given $d\in \N^*$, let $\Pd$ be the set of prime numbers up to $d$, then for each prime number $k\in \Pd$, there exists a polynomial $f_k(n)$ of variable $n$ with degree up to $d-1$ such that 
$$\det V(n,d) = \prod_{k \in \Pd} k^{f_k(n)}.$$
\end{conj}
For example, we can derive that $$\det V(n,5) = 2^{f_2(n)} 3^{f_3(n)} 5^{f_5(n)} =   2^{\frac{n\,\left(n-1\right)\,\left(n^2+3\,n+14\right)}{6}}\,3^{\frac{n\,\left(n-1\right)\,\left(n+1\right)}{2}}\,5^{\frac{n\,\left(n^3+10\,n^2+35\,n+74\right)}{24}}.$$
\begin{rem}
By taking large number of numerical tests, we still have no counterexample to reject the Conjectures \ref{conj:1}. If the conjecture is proved, then it is not difficult to derive the expression of $f_k(n)$ by the method of undetermined coefficients to obtain the expression of $\det V(n,d)$ for given $d$. The related topic deserves more attention in our future work.
\end{rem}

\bibliography{mybibfile}

\end{document}